\documentclass[12pt]{amsart}
\usepackage{amsmath,amssymb,amsbsy,amsfonts,amsthm,latexsym,mathabx,
            amsopn,amstext,amsxtra,euscript,amscd,stmaryrd,mathrsfs,
            cite,array,mathtools,enumerate}

\usepackage{url}
\usepackage[colorlinks,linkcolor=blue,anchorcolor=blue,citecolor=blue,backref=page]{hyperref}

\usepackage[norefs,nocites]{refcheck}
\usepackage{color}
\usepackage{float}

\hypersetup{breaklinks=true}

\usepackage[english]{babel}

\usepackage{mathtools}
\usepackage{todonotes}

\usepackage{enumitem}

\usepackage{mathtools}
\usepackage{todonotes}
\usepackage{url}
\usepackage[colorlinks,linkcolor=blue,anchorcolor=blue,citecolor=blue,backref=page]{hyperref}

\begin{document}

\newtheorem{theorem}{Theorem}
\newtheorem{lemma}[theorem]{Lemma}
\newtheorem{claim}[theorem]{Claim}
\newtheorem{cor}[theorem]{Corollary}
\newtheorem{prop}[theorem]{Proposition}
\newtheorem{definition}{Definition}
\newtheorem{question}[theorem]{Open Question}
\newtheorem{example}[theorem]{Example}
\newtheorem{remark}[theorem]{Remark}

\numberwithin{equation}{section}
\numberwithin{theorem}{section}

 \newcommand{\F}{\mathbb{F}}
\newcommand{\K}{\mathbb{K}}
\newcommand{\PP}{\mathbb{P}}
\newcommand{\D}[1]{D\(#1\)}
\def\scr{\scriptstyle}
\def\\{\cr}
\def\({\left(}
\def\){\right)}
\def\<{\langle}
\def\>{\rangle}
\def\fl#1{\left\lfloor#1\right\rfloor}
\def\rf#1{\left\lceil#1\right\rceil}
\def\le{\leqslant}
\def\ge{\geqslant}
\def\eps{\varepsilon}
\def\mand{\qquad\mbox{and}\qquad}

\def\vec#1{\mathbf{#1}}

\newcommand{\lcm}{\operatorname{lcm}}

\def\bl#1{\begin{color}{blue}#1\end{color}} 

\newcommand{\C}{\mathbb{C}}
\newcommand{\Fq}{\mathbb{F}_q}
\newcommand{\Fp}{\mathbb{F}_p}
\newcommand{\Disc}[1]{\mathrm{Disc}\(#1\)}
\newcommand{\Res}[1]{\mathrm{Res}\(#1\)}
\newcommand{\ord}{\mathrm{ord}}

\newcommand{\Q}{\mathbb{Q}}
\newcommand{\Z}{\mathbb{Z}}
\renewcommand{\L}{\mathbb{L}}

\newcommand{\Norm}{\mathrm{Norm}}

\def\cA{{\mathcal A}}
\def\cB{{\mathcal B}}
\def\cC{{\mathcal C}}
\def\cD{{\mathcal D}}
\def\cE{{\mathcal E}}
\def\cF{{\mathcal F}}
\def\cG{{\mathcal G}}
\def\cH{{\mathcal H}}
\def\cI{{\mathcal I}}
\def\cJ{{\mathcal J}}
\def\cK{{\mathcal K}}
\def\cL{{\mathcal L}}
\def\cM{{\mathcal M}}
\def\cN{{\mathcal N}}
\def\cO{{\mathcal O}}
\def\cP{{\mathcal P}}
\def\cQ{{\mathcal Q}}
\def\cR{{\mathcal R}}
\def\cS{{\mathcal S}}
\def\cT{{\mathcal T}}
\def\cU{{\mathcal U}}
\def\cV{{\mathcal V}}
\def\cW{{\mathcal W}}
\def\cX{{\mathcal X}}
\def\cY{{\mathcal Y}}
\def\cZ{{\mathcal Z}}

\def\fra{{\mathfrak a}} 
\def\frb{{\mathfrak b}}
\def\frc{{\mathfrak c}}
\def\frd{{\mathfrak d}}
\def\fre{{\mathfrak e}}
\def\frf{{\mathfrak f}}
\def\frg{{\mathfrak g}}
\def\frh{{\mathfrak h}}
\def\fri{{\mathfrak i}}
\def\frj{{\mathfrak j}}
\def\frk{{\mathfrak k}}
\def\frl{{\mathfrak l}}
\def\frm{{\mathfrak m}}
\def\frn{{\mathfrak n}}
\def\fro{{\mathfrak o}}
\def\frp{{\mathfrak p}}
\def\frq{{\mathfrak q}}
\def\frr{{\mathfrak r}}
\def\frs{{\mathfrak s}}
\def\frt{{\mathfrak t}}
\def\fru{{\mathfrak u}}
\def\frv{{\mathfrak v}}
\def\frw{{\mathfrak w}}
\def\frx{{\mathfrak x}}
\def\fry{{\mathfrak y}}
\def\frz{{\mathfrak z}}

\def\ov\Q{\overline{\Q}}
\def \brho{\boldsymbol{\rho}}

\def \fP {\mathfrak P}

\def \Prob{{\mathrm {}}}
\def\e{\mathbf{e}}
\def\ep{{\mathbf{\,e}}_p}
\def\epp{{\mathbf{\,e}}_{p^2}}
\def\em{{\mathbf{\,e}}_m}

\newcommand{\sR}{\ensuremath{\mathscr{R}}}
\newcommand{\sDI}{\ensuremath{\mathscr{DI}}}
\newcommand{\DI}{\ensuremath{\mathrm{DI}}}

\newcommand{\Orb}[1]{\mathrm{Orb}\(#1\)}
\newcommand{\aOrb}[1]{\overline{\mathrm{Orb}}\(#1\)}
\def \PrePer{{\mathrm{PrePer}}}
\def \Per{{\mathrm{Per}}}

\newcommand{\Tr}{\operatorname{Tr}}
\newcommand{\Nm}{\operatorname{Nm}}

\newenvironment{notation}[0]{%
  \begin{list}%
    {}%
    {\setlength{\itemindent}{0pt}
     \setlength{\labelwidth}{1\parindent}
     \setlength{\labelsep}{\parindent}
     \setlength{\leftmargin}{2\parindent}
     \setlength{\itemsep}{0pt}
     }%
   }%
  {\end{list}}

\definecolor{dgreen}{rgb}{0.,0.6,0.}
\def\tgreen#1{\begin{color}{dgreen}{\it{#1}}\end{color}}
\def\tblue#1{\begin{color}{blue}{\it{#1}}\end{color}}
\def\tred#1{\begin{color}{red}#1\end{color}}
\def\tmagenta#1{\begin{color}{magenta}{\it{#1}}\end{color}}
\def\tNavyBlue#1{\begin{color}{NavyBlue}{\it{#1}}\end{color}}
\def\tMaroon#1{\begin{color}{Maroon}{\it{#1}}\end{color}}

\title[Waring problem with 
Dickson polynomials]{On the Waring problem with 
Dickson polynomials modulo a prime}

 \author[I.~E.~Shparlinski]{Igor E. Shparlinski}
 \address{School of Mathematics and Statistics, University of New South Wales.
 Sydney, NSW 2052, Australia}
 \email{igor.shparlinski@unsw.edu.au}
 
  \author[J. F. Voloch] {Jos\'e Felipe Voloch}
\address{
School of Mathematics and Statistics,
University of Canterbury,
Private Bag 4800, Christchurch 8140, New Zealand}
\email{felipe.voloch@canterbury.ac.nz}

\begin{abstract}
We improve recent results of D.~Gomez and A.~Winterhof (2010)
and of A.~Ostafe and  I.~E.~Shparlinski (2011)
on the Waring  problem with Dickson polynomials in the case 
of prime finite fields. 
Our approach is based on recent bounds of  Kloosterman and Gauss sums
due to A.~Ostafe,  I.~E.~Shparlinski and J.~F.~Voloch (2021). 
\end{abstract}



\maketitle

\tableofcontents 
\section{Introduction}

\subsection{Previous results}
Let $\F_q$ be the finite field of $q$ elements. For  $a \in \F_q$ 
we define the sequence of  Dickson polynomials $D_e(X,a)$, 
$e =0,1, \ldots$, 
recursively by the relation
$$
D_e(X,a)  = XD_{e-1}(X,a) -a D_{e-2}(X,a), \qquad e=2, 3, \ldots, 
$$
where $D_0(X,a) = 2$ and $D_1(X,a) = X$, see~\cite{LMT} for 
background on Dickson polynomials. 

Gomez and Winterhof~\cite{GoWi} have considered an analogue of 
the Waring problem for Dickson polynomials over $\F_q$, that is, 
the question of the existence and estimation 
of a positive integer $s$ such that the equation 
\begin{equation}
\label{eq:War-Dick} 
D_e(u_1,a)+\ldots+ D_e(u_s,a) = c, \qquad u_1, \ldots, u_s \in \F_q,
\end{equation}
is solvable  for any $c \in \F_q$, see also~\cite{CMW}. 

In particular, we denote by $g_a(e,q)$ the smallest possible 
value of $s$ in~\eqref{eq:War-Dick} and put   $g_a(e,q) = \infty$
if such $s$ does not exist. 

Since for $a=0$ we have $D_e(X,a) = X^e$, this case corresponds to
the classical Waring problem in finite fields where recently 
quite substantial progress has been achieved, see~\cite{Cip,CCP,CoPi,vsWoWi};
a survey of earlier results can also be found in~\cite{Wint}.
So, we can restrict ourselves to the case of $a \in \F_q^*$. 


Using the identity 
\begin{equation}
\label{eq:Dick Ident} 
D_e(v+av^{-1},a) = v^e+a^ev^{-e},
\end{equation} 
which holds for any nonzero $v$ in the algebraic closure of $\F_q$,
see~\cite[Equation~(1.1)]{GoWi}, and   Weil-type bounds of 
additive character sums with rational functions, 
Gomez and Winterhof~\cite[Theorem~4.1]{GoWi} have proved that for $s \ge 3$ the 
inequality $g_a(e,q)  \le s$ holds 
\begin{itemize}
\item for any $a\in \F_q^*$ and $\gcd(e,q-1) \le  2^{-3} q^{1/2 -1/(2s-2)}$;
\item for $a=1$  and $\gcd(e,q+1) \le 2^{-1} q^{1/2 -1/(2s-2)}$.
\end{itemize}

Note that if $\min\{\gcd(e,q-1) , \gcd(e,q+1)\} \le   2^{-3} q^{1/4}$
the above result gives a very strong bound  $g_a(e,q)  \le 3$. 
Hence, throughout this paper we can assume that 
\begin{equation}
\label{eq:large e} 
\min\{\gcd(e,q-1) , \gcd(e,q+1)\}>  2^{-3} q^{1/4}.
\end{equation}

A different approach from~\cite{OS}, based on additive combinatorics, in particular on   results of  
Glibichuk~\cite{Glib}  and  Glibichuk 
and Rudnev~\cite{GlRud}
has allowed to substantially extend the range of $e$. 
In particular, by~\cite[Theorem~2]{OS} we have $g_a(e,q)  \le 16$ if 
\begin{itemize}
\item for any $a\in \F_q^*$ and 
$$
\gcd(e,q-1) \le 2^{-3/2} (q-2)^{1/2};
$$

\item  for any $a\in \F_q^*$ which is a square and 
$$
\gcd(e,q+1) \le 2^{-3/2} (q-2)^{1/2}. 
$$
\end{itemize}

Furthermore,  for any $\varepsilon > 0$, by~\cite[Theorem~2]{OS} we have an upper bound on $g_1(e,q)$ in terms of  
only $\varepsilon$, provided that 
$$
\min\{\gcd(e,q-1) , \gcd(e,q+1)\} \le   q^{1-\varepsilon}.
$$

We also note that a multivariate version of the above question has been studied in~\cite{OTW}. 

\subsection{Main results}
Here we use some results and also ideas from~\cite{OSV} to improve the above bounds in the case of prime 
$q=p$ and in  intermediate  ranges of $\gcd(e,p-1)$ and  $\gcd(e,p+1)$. 

Since we are mostly interested in large values of $e$, to simplify some technical 
details, we assume that $e$ is not very small.

\begin{theorem}
\label{thm:medium e}  Let $p$ be prime. There is an absolute constant $C > 0$ such 
that for any fixed even integer $s  \ge 4$, uniformly over  $a\in \F_p^*$ the 
inequality $g_a(e,p)  \le s$ holds provided that 
$$
 \gcd(e,p-1) \le C   p^{(4s -7)/(7s+8)} , 
$$
and, if $a$ is a quadratic residue modulo $p$,  also provided that 
$$
\gcd(e,p+1) \le   C  \max\left\{ p^{(11s-82)/(21s-42)},  p^{(6s-57)/(11s-22)}\right\}.
$$
\end{theorem}

Theorem~\ref{thm:medium e}  is based on new bounds on Kloosterman and Gauss sums
over a thin subgroup, see Lemmas~\ref{lem:Kloost} and~\ref{lem:Gauss} and is 
most interesting when  $\min\{\gcd(e,p-1) , \gcd(e,p+1)\}$ is large, for example,  
of order $p^{1/2}$ or slightly larger.  

Next,  we show that using the classical Weil bound, see~\eqref{eq:K-Weil}   below, 
we can still improve previous estimates in  certain ranges of   $\gcd(e,p-1)$ 
(below $p^{1/2}$). 

\begin{theorem}
\label{thm:small e}  Let $p$ be prime. There is an absolute constant $C > 0$ such 
that for any fixed even integer $s  \ge 4$, the
inequality $g_a(e,p)  \le s$ holds  provided that 
$$
 \gcd(e,p-1) \le  C  p^{1/2 - 1/(3s-8)}. 
$$
\end{theorem}

We note that the bound of Theorem~\ref{thm:medium e} with respect to 
$ \gcd(e,p-1) $ is significantly stronger than with respect to 
$ \gcd(e,p+1) $. Furthermore, while our proof of Theorem~\ref{thm:small e},
can easily be adjusted to work with $\gcd(e,p+1)$, in this case it merely recovers the previous 
result of Gomez and Winterhof~\cite[Theorem~4.1]{GoWi}. This is because we do 
not have a good version of  Lemma~\ref{lem:Energy F_p} below, see Question~\ref{quest:TraceEq}. 

However, for a related equation our approach works and leads to Lemma~\ref{lem:Energy F_p2}
which in turn allows us to get new results for the classical Waring problem with monomials 
in the norm-one subgroup of $\F_{p^2}$, that is in 
\begin{equation}
\label{eq:Norm1} 
\cN_{p^2} = \{z \in \F_{p^2}:~ \Nm(z) = 1\}, 
\end{equation} 
where $\Nm(z)= z^{p+1}$  is the $\F_{p^2}/\F_p$ norm of $z$.

Let $G(k,p)$  denote  the smallest possible 
value of $s$ such that the equation 
$$
 u_1^k +\ldots+ u_s^k = c, \qquad u_1, \ldots, u_s \in \cN_{p^2} ,
$$
is solvable  for any $c \in  \F_{p^2}$. 

Since we are mostly interested in large values of $k$ when traditional 
methods do not work, in order to simplify the calculations, we assume that 
$\gcd(k,p+1)  \ge p^{1/6}$. 

\begin{theorem}
\label{thm:monom med k}  Let $p$ be prime. There is an absolute constant $C > 0$ such 
that for any fixed even integer $s  \ge 4$,   the 
inequality $G(k,p) \le s$ holds provided that 
$$
p^{1/6}\le  \gcd(k,p+1) \le  C \min\left\{p^{(6s-186)/(11s-116)} ,  p^{(5s-56)/(10s-56)}\right\}. 
$$
\end{theorem}


\section{Preliminaries}

\subsection{Notation} 
We use $\# \cA$ to denote the
cardinality of  a finite set $\cA$. 

For a prime $p$ and $u \in \C$,  we let  
$$
\ep(u) = \exp(2 \pi i u/p).
$$ 

Finally, we recall that  the notations $U = O(V)$, $U \ll V$ and $ V\gg U$  
are equivalent to $|U|\leqslant c V$ for some positive constant $c$
which, throughout this work, is absolute.  

\subsection{Value sets of Dickson Polynomials} 
We note that through out the paper we use the following trivial observation
$$
\{v^e+a^ev^{-e} ~v \in \F_p\} \subseteq \{D_e(u,a):~u \in \F_p^*\} , 
$$
see~\eqref{eq:Dick Ident}.

Next, let  $\Tr(z) = z + z^p$  
be the trace of  $z\in \F_{p^2}$  in $\F_p$, 
respectively. 

We also recall the definition of the norm-one subgroup $\cN_{p^2}$ of $ \F_{p^2}^*$, 
given by~\eqref{eq:Norm1}. 
Then we note a slightly less obvious inclusion
$$
\{ b^e \Tr\(v^e\): ~ v \in \cN_{p^2}\} \subseteq \{D_e(u,a):~u \in \F_p\} 
$$
provided that $a = b^2$, $b\in \F_p^*$, is a quadratic residue modulo $p$.
Indeed,  we first note that for  $v \in \cN_{p^2}$
$$
bv+a(bv)^{-1}= b\(v+v^{-1}\)= b \Tr(v) \in \F_p.
$$
Thus 
$$
  \{D_e\(bv+a(bv)^{-1},a\):~v \in \cN_{p^2}\}  \subseteq \{D_e(u,a):~u \in \F_p\} .
  $$
On the other hand, from~\eqref{eq:Dick Ident} we have 
$$
D_e\(bv+a(bv)^{-1},a\) =  (bv)^e+a^e(bv)^{-e}
= b^e \Tr\(v^e\).
 $$


\subsection{Bounds of some exponential sums}

Given a multiplicative subgroup $\cH\subseteq \F_p^*$  and $\alpha, \beta \in \F_p$ we consider Kloosterman sums over $\cH$
$$
\cK_p(\cH; \alpha,\beta) = \sum_{u \in \cH} \ep\(\alpha u +\beta u^{-1}\). 
$$
The classical Weil bound for exponential sums with rational functions, see, for example,~\cite[Theorem~2]{MorMor}, implies 
\begin{equation}
\label{eq:K-Weil} 
\cK_p(\cH; \alpha,\beta)  \ll p^{1/2}
\end{equation} 
provided $(\alpha,\beta) \ne (0,0)$. 

The following bound is given by~\cite[Corollary~2.9]{OSV}. 

\begin{lemma}
\label{lem:Kloost}  Let  $p$ be prime  and let $\cH$  be a 
multiplicative subgroup of  $\F_p^*$ of order $\tau$. Then uniformly 
over $(\alpha,\beta) \in \F_p^2$, $(\alpha,\beta) \ne (0,0)$, we have
 $$
\cK_p(\cH; \alpha,\beta)  \ll \min\{p^{1/2}, \tau^{23/36} p^{1/6},  \tau^{20/27}   p^{1/9} \}. 
$$
\end{lemma}

We also need a similar result for Gaussian sums 
$$
\cG_p(\cH; \alpha) = \sum_{u \in \cH} \ep\(\Tr\(\alpha u\)\), 
$$
over $\F_{p^2}$.

We  recall the definition of $\cN_{p^2}$ in~\eqref{eq:Norm1}. 
We  need  the following analogue of~\eqref{eq:K-Weil},  following easily from 
the bound 
$$
 \sum_{u \in \cN_{p^2} } \chi(u) \ep\(\Tr\(\alpha u\)\) \ll p^{1/2}
$$
with an arbitrary multiplicative character $\chi$ of $\F_{p^2}^*$, which in turn is a very special 
case of  a result of Li~\cite[Theorem~2]{Li}.

\begin{lemma}
\label{lem:G-Weil}
 Let  $p$ be prime and let $\cH$ be  a
multiplicative subgroup of  $\cN_{p^2}$ of order $\tau$. Then uniformly 
over $\alpha  \in \F_{p^2}^*$,  we have
$$
\cG_{p^2} (\cH; \alpha)  \ll  p^{1/2}.
$$
\end{lemma}

Furthermore, by~\cite[Corollary~2.9]{OSV}, we have the following.

\begin{lemma}
\label{lem:Gauss}  Let  $p$ be prime and let $\cH$ be  a
multiplicative subgroup of  $\cN_{p^2}$ of order $\tau$. Then uniformly 
over $\alpha  \in \F_{p^2}^*$,  we have
$$
\cG_{p^2} (\cH; \alpha)  \ll  \min\left\{ \tau^{13/20} p^{1/6}, \tau^{34/45}   p^{1/9}, p^{1/2}  \right\}.
$$
\end{lemma}

\begin{proof} It has been shown in~\cite[Corollary~2.10]{OSV}. 
$$
\cG_{p^2} (\cH; \alpha)  \ll  \min\left\{\tau^{1/4} p^{1/2}, \tau^{13/20} p^{1/6}  , \tau^{34/45}   p^{1/9}  \right\}.
$$
However, it is easy to see that the first bound is always  dominated by the bound 
of Lemma~\ref{lem:G-Weil}. 
\end{proof}

We note that it is crucial for our improvement that the bounds of Lemmas~\ref{lem:Kloost}  
and~\ref{lem:Gauss} are nontrivial for $\tau < p^{1/2}$.

\subsection{Bounds on the number of solutions to some equations}

We first recall the following result, 
combining\cite[Theorem~(i)]{Vol} with the Weil bound~\cite[Equation~(5.7)]{Lor} 
which gives an upper bound 
on the number of points on curves over  $\F_p$. 

 \begin{lemma}
\label{lem:HighDeg} Let $p$ be prime and let 
$F(X,Y) \in \F_p[X,Y]$ be an absolutely irreducible polynomial of degree $d$.
Then
$$
\# \{(x,y) \in \F_p^2:~F(x,y) = 0\} \le 4d^{4/3} p^{2/3} + 3p. 
$$
\end{lemma}

We also need the following characterisation of possible factorisations of certain  
bivariate polynomials over the algebraic closure ${\bar{\F}}_p$ of $\F_p$.

\begin{lemma}
\label{lem:Irred}
Let  $p$ be prime  and let $F_e = X^{2e}Y^e + X^eY^{2e} + X^{e} + Y^{e} +A X^eY^e \in {\bar{\F}}_p[X,Y]$ with $p \nmid e$ and assume that $A \ne 0, \pm 4$. Then, for some $r\mid 8$, $F_e$ has $r$ absolutely irreducible factors of degree $3e/r$.
\end{lemma}

\begin{proof}
First we prove the following claim: If $F_1$ is irreducible and $K={\bar{\F}}_p(x,y)$ is the
function field of the curve defined by $F_1=0$, then the number of irreducible
factors of $F_e$ is the same as the index of the Galois group $\Gamma$ of $K(x^{1/e},y^{1/e})/K$ in the group $\mu_e \times \mu_e$ (where $x^{1/e},y^{1/e}$ are $e$-th roots of $x,y$ in some extension of $K$ and $\mu_e$ is the group of $e$-th roots of unity).

Indeed, if $G$ is the irreducible factor of $F_e$ with $G(x^{1/e},y^{1/e})=0$,
the Galois group $\Gamma$  preserves $G$, in the sense that
$G(\zeta X, \eta Y)= G(X,Y)$ for any  $(\zeta,\eta ) \in \Gamma$. On the other hand,
$F_e(\zeta X, \eta Y)= F_e(X,Y)$, for any $(\zeta,\eta ) \in \mu_e \times \mu_e$
so if we let 
$$
H(X,Y) = \prod_{(\zeta,\eta ) \in \Omega_e} G(\zeta X, \eta Y),
$$ with $(\zeta,\eta )$ running through the set $ \Omega_e$ of coset representatives of $\Gamma$ in $\mu_e \times \mu_e$, we have
$H\mid F_e$ and $H(\zeta X, \eta Y)= H(X,Y)$, $(\zeta,\eta ) \in \mu_e \times \mu_e$. It follows that $H=H_1(X^e,Y^e)$ and $H_1\mid F_1$. Since $F_1$ is irreducible by assumption, we get $H_1=F_1$,  $H=F_e$ and the claim follows.

It is now enough to show that the extension of $K$
obtained by adjoining $e$-th roots of $x$ and $y$ has degree $e^2/r$ over $K$ for some $r$ as in the statement of the lemma.
The proof now follows a similar strategy to the proof of~\cite[Lemma~4.3]{OSV}.

One can check directly, for example, with a computer algebra package, that $F_1$ is
absolutely irreducible if $A \ne 0$ and that $F_1=0$ defines a smooth projective curve if, in addition, $A \ne \pm 4$.





We   use repeatedly the elementary fact that $Z^e-c$, $c \in K$ has a factor of degree $d$ in $K[Z]$ if and only if $a$ is a $d$-th power in $K$ and $d \mid e$.
Indeed, if the factor is $\prod_{i \in I} (Z-\zeta^i c^{1/e})$, then the 
constant term is (up to a factor in ${\bar{\F}}_p$) equal to $c^{d/e}$ and the result follows.

Note that $[K:{\bar{\F}}_p(xy,x/y)]=2$. 
The divisor of the function $xy$ on $F_1=0$ is $2P-2Q$ where $P=(0,0)$ and $Q$ is the point at infinity on the line $X+Y=0$. So the polynomial $Z^e -xy$ is irreducible over $K$ for $e$ odd, 
by~\cite[Proposition~3.7.3]{Sti}. To treat the case of general $e$, it is then
enough to treat the case $e=2$.  
But $xy$ cannot be a square in $K$ as this
would give a function of divisor $P-Q$ which is impossible since $F_1=0$ defines a smooth projective curve of degree three, hence genus one so $K$ cannot have a function of degree one.

Let $z$ be a root of $Z^e -xy$ and consider the field $K(z)$. The function $x/y$ has divisor $4R-4S$ on the curve $F_1=0$ where $R,S$ are the points at infinity on the $y$-axis and $x$-axis respectively. It follows that the function $x/y$ is not an eighth-power or an odd power in $K(z)$ because $R,S$ are unramified in $K(z)/K$. So the extension of $K(z)$ obtained by adjoining a root of $W^e = x/y$ has degree $e, e/2$ or $e/4$ over $K(z)$ and the result follows.
\end{proof}


%

We use Lemmas~\ref{lem:HighDeg} and~\ref{lem:Irred}  to estimate  the number of solutions of the following  equation.

\begin{lemma}
\label{lem:Energy F_p}  Let  $p$ be prime  and let $\cH$  be a 
multiplicative subgroup of  $\F_p^*$ of order $\tau$. Then the number of solutions $R_\tau$ to the equation 
$$
u+u^{-1} + v+v^{-1} =  x+x^{-1} + y+y^{-1} , \qquad u,v,x,y \in \cH,
$$
satisfies 
$$
R_\tau \ll \tau^{8/3} + \tau^4/p.
$$
\end{lemma}

\begin{proof}   There are obviously at most $4 \tau$ choices of 
$(u,v) \in \cH^2 $ for which $u+u^{-1} + v+v^{-1} = 0, - 4$, and the also at most $O(\tau^2)$ pairs
$(x,y) \in \cH^2 $ which satisfy the above equation and thus we have $O(\tau^2)$ 
such solutions. 
We now fix $(u,v) \in \cH^2 $ such that for 
$A = -\(u+u^{-1} + v+v^{-1} \)$ we have $A  \ne 0, 4$. 
Clearly 
\begin{align*}
\# \{(x,y) \in \cH^2 :~ x+x^{-1} + y+y^{-1} - A = 0\}&\\
\le  e^{-2}\# \{(x,y) & \in \F_p^2:~F_e(x,y) = 0\} , 
\end{align*}
where   $e = (p-1)/\tau$ and $F_e(X,Y)$ is as in Lemma~\ref{lem:Irred}. 
Applying Lemma~\ref{lem:HighDeg} to each of at most 8 irreducible factors of 
$F_e$ each of degree at most $3e$, we obtain 
$$
R_\tau \ll \tau^{2} + e^{-2} \(e^{4/3} p^{2/3} + p\) \tau^2, 
$$
and the result follows. 
\end{proof} 

To improve  our main results for $\gcd(e,p+1)$ we need to obtain good estimates on the number of solutions 
to the following trace-equation.

\begin{question}
\label{quest:TraceEq} Given a subgroup $\cH \subseteq \cN_{p^2}$, obtain a version of Lemma~\ref{lem:Energy F_p} 
for the equation 
$$
\Tr(u + v) =  \Tr(x+y) , \qquad u,v,x,y \in \cH.
$$
\end{question}

While  Question~\ref{quest:TraceEq} is still open,  we 
are able  estimate  the so-called {\it additive energy\/} of a subgroup of  $\cN_{p^2}$.
We start with an analogue of~\cite[Lemma~4.5]{OSV}, which we believe is of independent 
interest.

\begin{lemma}
\label{lem:HighDeg-Fp2} Let $p$ be prime and let  $t= k(p-1)$, where $k$ is a positive
integer with $\gcd(k,p)=1$.
Then  for  $a \in \F_{p^2} $ with $a \ne 0$, for  the polynomial
$$
F(X,Y) =  X^t + Y^t +a \in \F_{p^2} [X,Y]
$$
we have
$$
\# \{(x,y) \in \F_{p^2}^2:~F(x,y) = 0\} \ll t^{6/5}p^{8/5} + {p^3}. 
$$ 
\end{lemma}

\begin{proof}  
The proof follows a strategy which is similar to that used in the proof of~\cite[Lemma~4.5]{OSV}. 

Clearly we can assume that
\begin{equation}
\label{eq:k small}
k < 2^{-5/4} p
\end{equation}
as otherwise the result is trivial.

We know that the equation $F(X,Y)=0$ defines a smooth, hence absolutely irreducible curve, of
degree $s$ that we   call $E$.
Let $\alpha \in \F_{p^2}$ satisfy $\alpha^t = -a$. The point $P_0=(0,\alpha)$
defines a point on $E$ and the line $Y=\alpha$ meets $E$ at $P_0$ with multiplicity $s$, since
$F(X,\alpha) = X^{t}$. We denote by $x,y$ the functions on $E$ satisfying $F(x,y)=0$.

We want to bound the number $R$ of solutions of $F=0$ in $\F_{p^2}$. We follow the proof of
Lemma~\ref{lem:HighDeg} given  in~\cite[Theorem~(i)]{Vol}. It proceeds by considering, for some
integer $m$, the embedding
of $E$ in $\PP^n$, {with $n = (m+2)(m+1)/2 - 1$}, given by the monomials in $X,Y$ of degree at most $m$.

We recall some definitions and some results from~\cite{StoVol}. We consider the given embedding of
$E$ in $\PP^n$. For a point $P \in E$, the order sequence of $E$ at $P$ is the sequence
$0=j_0<j_1<\cdots<j_n$ of all possible intersection multiplicities at $P$ of $E$ with a hyperplane in $\PP^n$.
The embedding is {\it classical} if the order sequence at a generic point of $E$ is $0,1,\ldots,n$ and non-classical,
otherwise. The point $P$ is an osculation point if the order sequence at $P$ is not $0,1,\ldots,n$ and a Weierstrass point if
the order sequence at $P$ is not the same as the order sequence at a generic point of $E$. These two notions
coincide if the embedding is classical. This embedding is $\F_q$-{\it Frobenius classical\/} if, for a generic point of $E$,
the numbers $0,1,\ldots,n-1$ are
the possible intersection multiplicities at $P$ of $E$ with a hyperplane in $\PP^n$ that also passes through the image
of $P$ under the $\F_q$-Frobenius map.

If this embedding is Frobenius classical,  then
by~\cite[Theorem~2.13]{StoVol}
\begin{equation}
\label{eq:stovol}
R \le (n-1)t(t-3)/2 + mt(p^2+n)/n.
\end{equation}

If
\begin{equation}
\label{eq:cond}
{m < p/2} \mand 
p \nmid \prod_{i=1}^m \prod_{j=-m}^{m-i} (ti + j)
\end{equation}
then we claim that
the above embedding is classical. Indeed, the order sequence of the embedding at the point $P_0$ defined above
consists of the integers ${ti+j}$, $i,j \ge 0$, $i+j \le m$ as follows by considering the order of vanishing at $P_0$ of the
functions $x^j(y-\alpha)^i$,  $i,j \ge 0$, $i+j \le m$. The claim now follows
from~\cite[Corollary~1.7]{StoVol}.

If the embedding is Frobenius classical, we get the inequality~\eqref{eq:stovol} as mentioned above. If the embedding is
classical but Frobenius nonclassical then, by~\cite[Corollary~2.16]{StoVol}, every rational point of $E$ is a
Weierstrass point for the embedding. Hence, as the embedding is classical, we get
\begin{equation}
\label{eq:wp}
R \le n(n+1)t(t-3)/2 + mt(n+1)
\end{equation}
{since the right-hand side is the number of Weierstrass points of the embedding counted with multiplicity, see~\cite[Page~6]{StoVol}}, 
Indeed, we note that, in the present case, the degree of the embedding (denoted by $d$ in~\cite{StoVol}) is $mt$
and the order sequence (denoted by $\varepsilon_i$ in~\cite{StoVol}) is just $\varepsilon_i = i$ since the embedding is classical,
thus $\varepsilon_1+\cdots+\varepsilon_n = n(n+1)/2$.  

We now choose
\begin{equation}
\label{eq:m opt}
m = \min\left\{\fl{(p/k)^{1/5}}, k-1\right\}.
\end{equation}
If $|i|,|j| \le m$, then for the choice of $m$ as in~\eqref{eq:m opt} we have
\begin{equation}
\label{eq:main ineq}
0 < |-ki + j| \le 2km \le  2 k^{4/5}p^{1/5} < p
\end{equation}
provided that~\eqref{eq:k small} holds.

Note also that $ti + j \equiv -ki + j \pmod p$, as $t=k(p-1)$. Hence,
$$
 \prod_{i=1}^m \prod_{j=-m}^{m-i} (ti + j) \equiv 
  \prod_{i=1}^m \prod_{j=-m}^{m-i} (-ki + j) \pmod p.
$$
Thus from the definition of $m$ in~\eqref{eq:m opt} and the inequalities~\eqref{eq:main ineq} we see that the conditions~\eqref{eq:cond} are satisfied.
We note that~\eqref{eq:stovol}  and~\eqref{eq:wp}   can be simplified and
combined as
$$
R   \ll \max\{m^2 t^2 + tp^2/m,  m^4 t^2\}
\ll tp^2/m +  m^4 t^2.
$$ 

Since $m \ll (p/k)^{1/5}\ll (p^2/t)^{1/5}$, we have $m^4 t^2 \ll tp^2/m$ and thus
we obtain
$$
R \ll tp^2/m {\ll k p^3/m}.
$$
Recalling the choice of $m$ in~\eqref{eq:m opt}, we obtain the desired result.
\end{proof}

We note that for $k \gg p^{1/6}$ the bound of Lemma~\ref{lem:HighDeg-Fp2}
is $O\(t^{6/5}p^{8/5}\)$.

Using Lemma~\ref{lem:HighDeg-Fp2} instead of Lemma~\ref{lem:HighDeg} 
(and noticing that $\tau = (p^2-1)/t$), 
we derive the following analogue of Lemma~\ref{lem:Energy F_p}. 

\begin{lemma}
\label{lem:Energy F_p2}  Let  $p$ be prime  and let $\cH$  be a 
multiplicative subgroup of  $\cN_{p^2}$ of order $\tau$. Then the  number of solutions $T_\tau$ to the equation 
\[
u  + v  =  x  + y , \qquad u,v,x,y \in \cH,
\]
satisfies 
\[
T_\tau \ll \tau^{14/5} + \tau^4/p.
\]
\end{lemma}


\section{Proof of Theorem~\ref{thm:medium e}}

\subsection{Small $\gcd(e,p-1)$}
\label{sec: p-1}
Let 
$$
\tau = \frac{p-1}{\gcd(e,p-1)},
$$
and let $\cH$ be the subgroup of $\F_p$ of order $\tau$. By our assumption~\eqref{eq:large e}  on $e$, we have
\begin{equation}
\label{eq: small tau} 
 \tau \ll p^{3/4}  .
\end{equation}  

We write $s = 2r$ and  denote by $\cF_r$ the set of $f \in \F_p$ which cannot be 
represented as 
$$
f = \sum_{i=1}^r \(u_i+ au_i^{-1}\), \qquad u_i \in \cH, \ i =1, \ldots, s.
$$
In particular, for the number $N_r(\cF_r)$ of the solutions to the   equation
$$
f = \sum_{i=1}^r \(u_i+ au_i^{-1}\), \qquad  f \in \cF_r, \ u_i \in \cH, \ i =1, \ldots, s, 
$$  
we have $N_r(\cF_r) = 0$. 
We see from~\eqref{eq:Dick Ident}  that each element of the complementing set $\cR_r = \F_p \setminus \cF_r$
can be represented by a sum of  $s$ values of $D_e(x,a)$, $x \in \F_p$.

On the other hand, by the orthogonality of exponential functions, we have
\begin{align*}
N_r(\cF_r) & = \sum_{u_1, \ldots, u_r \in \cH} \sum_{ f \in \cF_r}  \frac{1}{p} \sum_{\alpha \in \F_p} 
\ep\(\alpha\(\sum_{i=1}^r \(u_i+ au_i^{-1}\) - f\)\)\\
 & =  \frac{1}{p} \sum_{\alpha \in \F_p} \(\sum_{u \in \cH} \ep\(\alpha \(u+ au^{-1}\)\)\)^r  \sum_{f \in \cF_r} \ep\(-\alpha f \)\\
& = \frac{1}{p}( \#\cH)^r \#\cF_r + O\(p^{-1}\Delta\), 
\end{align*}
where 
$$
\Delta =  \sum_{\alpha \in \F_p^*}  \left| \sum_{u \in \cH} \ep\(\alpha \(u+ au^{-1}\) \)
\right|^r  \left| \sum_{f \in \cF_r} \ep\(\alpha f \)\right|.
$$
Thus, recalling that $N_r(\cF_r) = 0$, we obtain 
\begin{equation}
\label{eq:Ns Delta} 
\tau^r \#\cF_r  \ll \Delta
\end{equation}

Next, using the second  bound of Lemma~\ref{lem:Kloost}, we write 
\begin{equation}
\label{eq: DeltaGamma-1} 
\Delta \ll \(\tau^{20/27} p^{1/9}\)^{r-2} \Gamma, 
\end{equation}
where (after extending the summation to all $\alpha \in \F_p$) we can take
$$
\Gamma =  \sum_{\alpha \in \F_p}  \left| \sum_{u \in \cH} \ep\(\alpha \(u+ au^{-1}\) \)
\right|^2  \left| \sum_{f \in \cF_r} \ep\(\alpha f \)\right|.
$$

By the Cauchy inequality 
$$
\Gamma^2  \le  \sum_{\alpha \in \F_p}  \left| \sum_{u \in \cH} \ep\(\alpha \(u+ au^{-1}\) \)
\right|^4    \sum_{\alpha \in \F_p}  \left| \sum_{f \in \cF_r} \ep\(\alpha f \)\right|^2.
$$

Using the orthogonality of exponential functions again, and also Lemma~\ref{lem:Energy F_p}, 
and recalling~\eqref{eq: small tau}, we infer 
$$
\Gamma^2 \ll \(p \tau^{8/3} +\tau^4\) p \#\cF_r \ll p^2  \tau^{8/3}  \#\cF_r. 
$$
We now see from~\eqref{eq:Ns Delta}  and~\eqref{eq: DeltaGamma-1} that 
\begin{equation}
\label{eq: Main Ineq -1} 
\tau^r \#\cF_r  \ll \(\tau^{20/27} p^{1/9}\)^{r-2} p  \tau^{4/3}\( \#\cF_r\)^{1/2}, 
\end{equation}
or
\begin{align*}
 \#\cF_r  &\ll \tau^{40r/27-80/27+8/3-2r}  p^{2(r-2)/9 + 2} =
 \tau^{-14r/27-8/27} p^{2r/9+14/9} \\
 &\ll \gcd(e,p-1)^{14r/27+8/27} p^{-8r/27+34/27}. 
\end{align*}
Therefore, there is an absolute constant $C> 0$ such that for 
$$
\gcd(e,p-1) \le C p^{(8r-7)/(14r+8)}
$$
we have  $\#\cF_r < p/2$.

Thus, for any $f \in \F_p$ we see that the set  $f - \cR_r= \{f- u:~u \in \cR_r\}$ of cardinality 
 $\#\cR_r > p/2$ has a nontrivial intersection with $\cR_r$. Hence $f \in \cR_{2r}$.


 \begin{remark} Certainly the first bound of Lemma~\ref{lem:Kloost} can also be used
 in our argument. However the bound it implies is always weaker than a combination 
 of the current bound and the bound of Theorem~\ref{thm:small e}.
 \end{remark}

\subsection{Small $\gcd(e,p+1)$}
We now let 
$$
\tau = \frac{p+1}{\gcd(e,p+1)},
$$
and let $\cH$ be the subgroup of $\cN_{p^2} $ of order $\tau$.

This time we write $s = 2r$ and  denote by $\cF_r$ the set of $f \in \F_p$ which cannot be 
represented as 
$$
f = \sum_{i=1}^r u_i+ u_i^{-1} =   \sum_{i=1}^r \Tr(u_i) , \qquad u_i \in \cH, \ i =1, \ldots, s.
$$
In particular, for the number $N_r(\cF_r)$ of the solutions to the   equation
$$
f =   \sum_{i=1}^r \Tr(u_i), \qquad  f \in \cF_r, \ u_i \in \cH, \ i =1, \ldots, s, 
$$  
we have $N_r(\cF_r) = 0$. 
We see from~\eqref{eq:Dick Ident}  that each element the complementing set $\cR_r = \F_p \setminus \cF_r$
can be represented by a sum of  $s$ values of $D_e(x,a)$, $x \in \F_p$

On the other hand, by the orthogonality of exponential functions, we have
\begin{align*}
N_r(\cF_r) & = \sum_{u_1, \ldots, u_r \in \cH} \sum_{ f \in \cF_r}
 \frac{1}{p} \sum_{\alpha \in \F_p} \(\sum_{u \in \cH} \ep\(\alpha  \Tr(u\)\)^r  \sum_{f \in \cF_r} \ep\(-\alpha f \)\\
 & = 
 \frac{1}{p} \sum_{\alpha \in \F_p} \ep\(\alpha\(\sum_{i=1}^r\Tr(u_i) - f\)\)\\
 & = \frac{1}{p}( \#\cH)^r \#\cF_r + O(p^{-1}\Delta), 
\end{align*}
where 
$$
\Delta =  \sum_{\alpha \in \F_p^*}  \left| \sum_{u \in \cH} \ep\(\alpha  \Tr(u) \)
\right|^r  \left| \sum_{f \in \cF_r} \ep\(\alpha f \)\right|.
$$
Thus, recalling that $N_r(\cF_r) = 0$, we obtain a full analogue of~\eqref{eq:Ns Delta}.

Next, using the first bound of  Lemma~\ref{lem:Gauss}, we write 
\begin{equation}
\label{eq: DeltaGamma-2} 
\Delta \ll \(\tau^{13/20} p^{1/6}\)^{r-1} \Gamma, 
\end{equation}
where (after extending the summation to all $\alpha \in \F_p$) we can take
$$
\Gamma =  \sum_{\alpha \in \F_p}  \left| \sum_{u \in \cH} \ep\(\alpha  \Tr(u)  \)
\right|  \left| \sum_{f \in \cF_r} \ep\(\alpha f \)\right|.
$$
We remark that in Section~\ref{sec: p-1},   Lemma~\ref{lem:Kloost} is  used $r-2$ times, 
while now the bounds of   Lemma~\ref{lem:Gauss} is used  $r-1$ times. This is because we are lacking an appropriate 
analogue of Lemma~\ref{lem:Energy F_p} in this case.  

Next by the Cauchy inequality 
$$
\Gamma^2  \le  \sum_{\alpha \in \F_p}  \left| \sum_{u \in \cH} \ep\(\alpha \Tr(u)  \)
\right|^2    \sum_{\alpha \in \F_p}  \left| \sum_{f \in \cF_r} \ep\(\alpha f \)\right|^2. 
$$

We observe that for elements $u.v \in \cN_{p^2}$, the equation $ \Tr(u)  = \Tr(v)$ 
implies that $u$ and $v$ have the same characteristic polynomial. 
Thus, either $u=v$ or $u = v^p = v^{-1}$ and by the orthogonality of exponential functions again, 
we derive 
$$
\Gamma^2 \ll p^2 \#\cH \#\cF_r \ll p^2  \tau  \#\cF_r. 
$$

We now see from~\eqref{eq:Ns Delta}  and~\eqref{eq: DeltaGamma-2} that 
\begin{equation}
\label{eq: Main Ineq +1} 
\tau^r \#\cF_r  \ll  \(\tau^{13/20} p^{1/6}\)^{r-1}  p  \tau \( \#\cF_r\)^{1/2}, 
\end{equation}
or
\begin{align*}
 \#\cF_r  &\ll \tau^{13r/10-13/10+2-2r}  p^{(r-1)/3 + 2} =
 \tau^{-7(r-1)/10} p^{r/3+5/3} \\
 &\ll \gcd(e,p+1)^{7(r-1)/10} p^{-(11 r-41)/30}. 
\end{align*}
Therefore, there is an absolute constant $C> 0$ such that for 
$$
\gcd(e,p+1) \le C p^{(11r-41)/(21r-21)}
$$
we have  $\#\cF_r < p/2$ and we conclude the proof as before. 
 
 Similarly, using the second bound of Lemma~\ref{lem:Gauss}, instead of~\eqref{eq: Main Ineq +1} 
 we derive that 
 $$
 \tau^r \#\cF_r  \ll \(\tau^{34/45}   p^{1/9} \)^{r-1} p  \tau \( \#\cF_r\)^{1/2}.
 $$
 Then after simple calculations we obtain that 
 $\cR_{2r} = \F_p$ provided 
$$
\gcd(e,p+1) \le C p^{(12r-57)/(22r-22)}. 
$$

 \section{Proof of Theorem~\ref{thm:small e}}

We proceed as in the proof of Theorem~\ref{thm:medium e}.
Indeed, an application of the bound~\eqref{eq:K-Weil},  instead of~\eqref{eq: Main Ineq -1}  leads us to the inequality
$$
\tau^r \#\cF_r  \ll \(p^{1/2}\)^{r-2} p  \tau^{4/3}\( \#\cF_r\)^{1/2}. 
$$
This implies that 
 $\cR_{2r} = \F_p$ provided 
$$
\gcd(e,p-1)  \le C p^{(3r-5)/(6r-8)} 
$$
and we obtain the desired result with $s = 2r$. 


\section{Proof of Theorem~\ref{thm:monom med k}}

Clearly, without loss of generality we can assume that $k \mid p+1$. 
Clearly the set of powers $x^k$, $x \in \cN_{p^2}$ forms a subgroup $\cH$
of  $\cN_{p^2}$ of order 
$$
\tau = (p+1)/k.
$$
Note that since $k \ge p^{1/6}$ we get $\tau \ll p^{5/6}$. In this case the bounds of 
 Lemma~\ref{lem:Energy F_p2} takes forme 
$$
T_\tau \ll \tau^{14/5}.
$$

Similarly to the proof of Theorem~\ref{thm:medium e}, we write $s = 2r$ and  denote by $\cG_r$ the set of $f \in \F_{p^2}$ which can not be 
represented as 
$$
f = \sum_{i=1}^r u_i, \qquad u_i \in \cH, \ i =1, \ldots, s.
$$

The previous argument used with Lemmas~\ref{lem:Gauss} and~\ref{lem:Energy F_p2}
gives the following analogues of~\eqref{eq: Main Ineq -1} 
$$
\tau^r \#\cG_r  \ll      \begin{cases} 
\( \tau^{34/45}   p^{1/9}\)^{r-2} p^2 \tau^{14/5} \( \#\cG_r\)^{1/2},\\
 \( p^{1/2}\)^{r-2} p^2 \tau^{14/5}   \( \#\cG_r\)^{1/2}.
 \end{cases}
$$
We note that here we do not use the first bound of Lemma~\ref{lem:Gauss} 
as it does not seems to give anything better than a combination of the 
other two. 

These inequalities imply that for an appropriate absolute constant $C$
if one of the conditions 
$$
k \le C  \begin{cases} 
 p^{(6r-93)/(11r-58)} , \\
 p^{(5r-28)/(10r-28)}, 
 \end{cases}
$$
is satisfied, then $\# \cG_r < p^2/2$ and the previous argument concludes the proof with $s = 2r$.

%

\section*{Acknowledgement} 

The authors are grateful to Arne Winterhof for very helpful comments and 
suggestions.

During the preparation of this work 
I.E.S. was  partially supported by ARC Grants
DP230100530 and DP230100534, 
J.F.V.  by the Ministry for Business, Innvovation and Employment and by the Marsden Fund, administered by the Royal Society of New Zealand.

\end{document}